\documentclass[11pt]{article}

 \usepackage{array}
 \usepackage{enumitem}
\newcolumntype{C}[1]{>{\centering\arraybackslash}m{#1}}
\usepackage{subfig}
\usepackage{graphicx}
\usepackage{amssymb}
\usepackage{mathtools}
\usepackage{amsmath}
\usepackage{amsthm}
\usepackage{latexsym}
\usepackage{amsfonts}
\usepackage{etoolbox}
\usepackage{multirow}
\usepackage{float}
\usepackage{amsmath,color}
\usepackage{url}
\usepackage{hyperref}
\usepackage[T1]{fontenc}

\usepackage{arydshln}
\usepackage{lipsum}

\theoremstyle{plain}
\newtheorem{theorem}{Theorem}{}
\newtheorem{lemma}{Lemma}{}
{}
\newtheorem{corollary}{Corollary}{}
{}
{}
{}
\theoremstyle{definition}
\newtheorem{definition}{Definition}

\theoremstyle{definition}
\newtheorem{example}[theorem]{Example}
\usepackage[utf8]{inputenc}
\usepackage{authblk}

\title{\bf Gain distance Laplacian matrices for complex unit gain graphs}
\author{Suliman Khan}
\affil{Department of Mathematics and Physics,\\
University of Campania "Luigi Vanvitelli",\\
Viale Lincoln 5, Caserta, I–81100, Italy \vspace{0.4cm}\\

Department of Econometrics and O.R.,\\
Tilburg University, Tilburg, Netherlands \vspace{0.3cm}\\

 \small{Email: suliman.khan@unicampania.it \\ ~~~~~~~~~suliman5344@gmail.com }}

\date{}
\begin{document}

\maketitle

\vspace{-0.1cm}


\begin{abstract}  \noindent
A complex unit gain graph (or a $\mathbb{T}$-gain graph) $\Theta(\Sigma,\varphi)$ is a graph where the unit complex number is assign by a function $\varphi$ to every oriented edge of $\Sigma$ and assign its inverse to the opposite orientation. In this paper, we define the two gain distance Laplacian matrices $DL^{\max}_{<}(\Theta)$ and $DL^{\min}_{<}(\Theta)$ corresponding to the two gain distance matrices $D^{\max}_{<}(\Theta)$ and $D^{\min}_{<}(\Theta)$ defined for $\mathbb{T}$-gain graphs $\Theta(\Sigma,\varphi)$, for any vertex ordering $(V(\Sigma),<)$. Furthermore, we provide the characterization of singularity and find formulas for the rank of those Laplacian matrices. We also establish two types of characterization for balanced in complex unit gain graphs while using the gain distance Lapalcian matrices. Most of the results are derived by proving them more generally for weighted $\mathbb{T}$-gain graphs.

\end{abstract}
\vspace{0.2cm}
\thanks{{\em Math. Subj. Class.}  05C05, 05C50}\\
\thanks{{\em Keywords}: $\mathbb{T}$-gain graph, gain distances matrices, gain distance Laplacian matrices, singularity, balanced, rank, weighted $\mathbb{T}$-gain graphs}

\section{Introduction}
\indent
In graph theory, a gain graph or a $\mathcal{G}$-gain graph typically refers to a graph where each edge is labeled by a gain. Let $\Sigma=(V,E)$ be an undirected graph with a vertex set $V$ and edge set $E$ and let denote the set of oriented edges in $\Sigma$ by $\overrightarrow{E}(\Sigma)$. The {\it gain graph} define on a group $\mathcal{G}$ is a triple $\Theta=(\Sigma,\mathcal{G},\varphi)$, where $\Sigma$ is the underlying graph, $\mathcal{G}$ is a gain group (an abstract group) and  $\varphi:\overrightarrow{E}(\Sigma)\to \mathcal{G}$ is a map that for every oriented edge of $\Sigma$ we have $\varphi(e_{jk})=\varphi(e_{kj})^{-1}$, known as gain function. Connivently, we write $\Theta=(\Sigma,\varphi)$ for a $\mathcal{G}$-gain graph. Let $\mathbb{T}$ be a {\it circle group}:  $\mathbb{T}=\{ z \in \mathbb{C}:|z|=1 \}$ a subgroup of multiplicative group of non-zero complex numbers. For a particular choice of $\mathbb{T}$ in $\Theta=(\Sigma,\mathcal{G},\varphi)$ a gain graph is said to be {\it complex unit gain graph} or {\it $\mathbb{T}$-gain graph}. The {\it signed graph} is a complex unit gain graph with gains $\{+1,-1\}$ and the undirected graphs be considered a complex unit gain graph with gain $+1$. To study more about complex unit gain graphs we refer the readers to \cite{Reff2012, Belardo2023, Khan2024}.\\

The study of matrices and eigenvalues play a crucial role in graph theory by providing powerful tools for analyzing the structure, connectivity, and properties of graphs, as well as for developing graph algorithms and techniques. Associated to a graph the researchers defined different types of matrices; like adjacency matrix, Laplacian matrix, singless Laplacian matrix, Normalized Lapalcian matrix, distance matrix, incidence matrix etc. The distance matrices is a well-studied class of matrices among them. For a graph $\Sigma=(V,E)$ the distance between any two vertices $v_j, v_k \in V(\Sigma)$ is denoted by $d(v_j,v_k)$ and define to be the shortest path length between $v_j$ and $v_k$. In \cite{Edelberg1976}, Edelberg et al., studied the distnce matrices for trees. The distance matrices for directed graphs were studied by Graham et al., in \cite{Graham1977}. In \cite{Graham1978}, Lovász and Graham derived an expression for the inverse of a distance matrix of a tree. This concept further extended to the weighted trees by Bapt et al., in \cite{Bapat2005}. To read more about distance matrices we recommend the readers to study \cite{Aouchiche2014}; a survey on distances matrices and their spectra by Aouchiche and Hansen.\\

Recently, Hammed et al., introduced the two types of signed distance matrices for signed graphs \cite{Hameed2021}. This work is extended by Samanta et al., to the setting of complex unit gain graphs \cite{Samanta2022}, and they defined the two gain distances matrices $D^{\max}_{<}(\Theta)$ and $D^{\min}_{<}(\Theta)$ for a $\mathbb{T}$-gain graph $\Theta(\Sigma,\varphi)$, for any vertex ordering $(V(\Sigma),<)$. In this paper, we define the two gain distance Laplacian matrices $DL^{\max}_{<}(\Theta)$ and $DL^{\min}_{<}(\Theta)$ corresponding to those two gain distance matrices of $\mathbb{T}$-gain graphs. Furthermore, we provide the characterization of singularity and find formulas for the rank of these matrices. We also establish two types of characterization for balanced in complex unit gain graphs while using the gain distance Lapalcian matrices. Most of the results are derived by proving them more generally for weighted $\mathbb{T}$-gain graphs.

\subsection{Gain distance matrices}
In \cite{Samanta2022}, Samanta and Kannan introduced the concept of gain distance matrices of complex unit gain graphs. Obviously, this was the generalization to the previously defined two concepts: the distance matrices defined for undirected graphs and the signed distance matrices defined for signed graphs.

\begin{definition}\label{def1}
\cite{Samanta2022} Let $\Theta=(\Sigma,\varphi)$ denote a connected complex unit gain graph on $\Sigma$. For $y,z \in V(\Sigma)$, the oriented path form $y$ to $z$ is denoted by $yPz$. Then, the three types of path are defined as follows:\\

(P1)~ $\mathcal{P}(y,z)=\{yPz: ~ \text{where} ~yPz$ represent a shortest path\};\\

(P2)~ $\mathcal{P}^{\max}(y,z)=\{yPz \in \mathcal{P}(y,z): ~ Re(\varphi(yPz))=\max\limits_{y\tilde{P}z \in \mathcal{P}(y,z)}Re(\varphi(y\tilde{P}z))$ \};\\

(P3)~ $\mathcal{P}^{\min}(y,z)=\{yPz \in \mathcal{P}(y,z): ~ Re(\varphi(yPz))=\min\limits_{y\tilde{P}z \in \mathcal{P}(y,z)}Re(\varphi(y\tilde{P}z))$ \};
\end{definition}

Denote the order of a vertex set in $\Sigma$ by $(V(\Sigma,<))$; where the symbol $<$ is the standard ordering of vertices if $v_1<v_2<v_3< \cdots v_n$. For any vertex set ordering $<$, the reverse ordering is denoted by symbol $<_{r}$ and defined as: $v_j<_{r}v_{k}$ if and only if $v_{k}<v_{j}$, for any $j,k$. Next, we define the auxiliary gains.

 \begin{definition}\label{def2}
\cite{Samanta2022} Let $\Theta=(\Sigma,\varphi)$ denote a connected complex unit gain graph with the vertex ordering $(V(\Sigma),<)$. Then, the two types of auxiliary gains are defined:\\

(A1). $\varphi^{\max}_{<}$ is a maximum auxiliary gain defined by a map $\varphi^{\max}_{<}:V(\Sigma)\times V(\Sigma)\rightarrow \mathbb{T}$, satisfying the following two properties.\\
(i) $\varphi^{\max}_{<}(y,y)=0$ for all $y \in V(\Sigma)$;\\
(ii) Whenever $y<z$, $\varphi^{\max}_{<}(y,z)=\varphi(yPz)$, where $yPz$ represent an oriented path from $y$ to $z$ that is $yPz \in \mathcal{P}^{\max}(y,z)$ with $Im(\varphi(yPz))=\max\limits_{y\tilde{P}z \in \mathcal{P}^{\max}(y,z)}Im(\varphi(y\tilde{P}z))$. Moreover, $\varphi^{\max}_{<}(z,y)=\overline{\varphi^{\max}_{<}(y,z)}$.\\

(A2). $\varphi^{\min}_{<}$ is a minimum auxiliary gain defined by a map $\varphi^{\min}_{<}:V(\Sigma)\times V(\Sigma)\rightarrow \mathbb{T}$, satisfying the following two properties.\\
(i) $\varphi^{\min}_{<}(y,y)=0$ for all $y \in V(\Sigma)$;\\
(ii) Whenever $y<z$, $\varphi^{\min}_{<}(y,z)=\varphi(yPz)$, where $yPz$ represent an oriented path from $y$ to $z$ that is $yPz \in \mathcal{P}^{\min}(y,z)$ with $Im(\varphi(yPz))=\min\limits_{y\tilde{P}z \in \mathcal{P}^{\min}(y,z)}Im(\varphi(y\tilde{P}z))$. Moreover, $\varphi^{\min}_{<}(z,y)=\overline{\varphi^{\min}_{<}(y,z)}$.\\

Remember that, for $y<z$, $\varphi^{\min}_{<}(y,z)$ (resp. $\varphi^{\min}_{<}(y,z))$ is the minimum (resp. maximum) gain of $\mathcal{P}(y,z)$ (shortest paths from $y$ to $z$) as regards the lexicographical order of complex number $\mathbb{C}$ (i.e., $u+iv<w+ix$ if either $u<w$ or if $u=w$ then $v<x$). That is,\\

(a) If $y<z$, then $\varphi^{\max}_{<}(y,z)=\max\limits_{yPz \in \mathcal{P}(y,z)}\varphi(yPz)$. Further, $\varphi^{\max}_{<}(z,y)=\overline{\varphi^{\max}_{<}(y,z)}$.\\

(b) If $y<z$, then $\varphi^{\min}_{<}(y,z)=\min\limits_{yPz \in \mathcal{P}(y,z)}\varphi(yPz)$. Further, $\varphi^{\min}_{<}(z,y)=\overline{\varphi^{\min}_{<}(y,z)}$.

\end{definition}

 \begin{definition}\label{def3}
 (\textbf{Gain distances} \cite{Samanta2022}).\\
 Let $\Theta=(\Sigma,\varphi)$ denote a connected complex unit gain graph with the vertex ordering $(V(\Sigma),<)$. Then, for $y,z \in V(\Sigma)$, two gain distances are defined by follows: \\

(i) $d^{\max}_{<}(y,z)=\varphi^{\max}_{<}(y,z)d(y,z)$;\\

(ii) $d^{\min}_{<}(y,z)=\varphi^{\min}_{<}(y,z)d(y,z)$;\\

where, $d(y,z)$ is the distance from vertex $y$ to vertex $z$.
\end{definition}

\begin{definition}\label{def4}
 (\textbf{Gain distance matrices} \cite{Samanta2022}).\\
 Let $\Theta=(\Sigma,\varphi)$ denote a connected complex unit gain graph with the vertex ordering $(V(\Sigma),<)$. Consider, the vertex set of $\Sigma$ is $V(\Sigma)=\{v_1, \cdots, v_n\}$, then the two gain distance matrices associated with a vertex ordering $(V(\Sigma),<)$ are defined by follows: \\

(i) $D^{\max}_{<}(\Theta)=d^{\max}_{<}(v_s,v_t)$;\\

(ii) $D^{\min}_{<}(\Theta)=d^{\min}_{<}(v_s,v_t)$.\\

Here, $d^{\max}_{<}(v_s,v_t)$ and $d^{\min}_{<}(v_s,v_t)$ are the $(s,t)$-th entries of $D^{\max}_{<}(\Theta)$ and $D^{\min}_{<}(\Theta)$, respectively.
\end{definition}

The following example illustrate the above definitions.
\begin{example}
Let $\Theta=(\Sigma,\varphi)$ denote a connected complex unit gain graph with the standard vertex ordering $(V(\Sigma),<)$, as shown in Figure \ref{Fig1}. Then $D^{\max}_{<}(\Theta)$ and $D^{\min}_{<}(\Theta)$ for a graph $\Theta=(\Sigma,\varphi)$ with the given gains defined in Figure \ref{Fig1}, are presented in the following way, respectively.

$$D^{\max}_{<}(\Theta)=\left(
              \begin{array}{ccccc}
                0 & 1 & 2e^{i\pi/4} & 1 & e^{i\pi/4}\\
                1 & 0 & e^{i\pi/4} & 2 & 2e^{i\pi/4}\\
                2e^{-i\pi/4} & e^{-i\pi/4} & 0 & e^{i\pi/4} & 3\\
                1 & 2 & e^{-i\pi/4} & 0 & 2e^{i\pi/4}\\
                e^{-i\pi/4} & 2e^{-i\pi/4}& 3 & 2e^{-i\pi/4} & 0\\
              \end{array}
            \right).
$$
\\
Now, let $<_r$ be the reverse odering of the given standard ordering $<$, then we have\\

$$D^{\max}_{<_r}(\Theta)=\left(
              \begin{array}{ccccc}
                0 & 1 & 2e^{-i\pi/4} & 1 & e^{i\pi/4}\\
                1 & 0 & e^{i\pi/4} & 2 & 2e^{i\pi/4}\\
                2e^{i\pi/4} & e^{-i\pi/4} & 0 & e^{i\pi/4} & 3\\
                1 & 2 & e^{-i\pi/4} & 0 & 2e^{i\pi/4}\\
                e^{-i\pi/4} & 2e^{-i\pi/4}& 3 & 2e^{-i\pi/4} & 0\\
              \end{array}
            \right).
$$
\\

Since $D^{\max}_{<}(\Theta) \neq D^{\max}_{<_r}(\Theta)$. This implies that, $\text{spec}(D^{\max}_{<}(\Theta)) \neq \text{spec}(D^{\max}_{<_r}(\Theta))$. The same situation hold for the minimum case, i.e., $D^{\min}_{<}(\Theta) \neq D^{\min}_{<_r}(\Theta)$.
\end{example}
\begin{figure}[htbp!]
\centering
  \includegraphics[width=7cm]{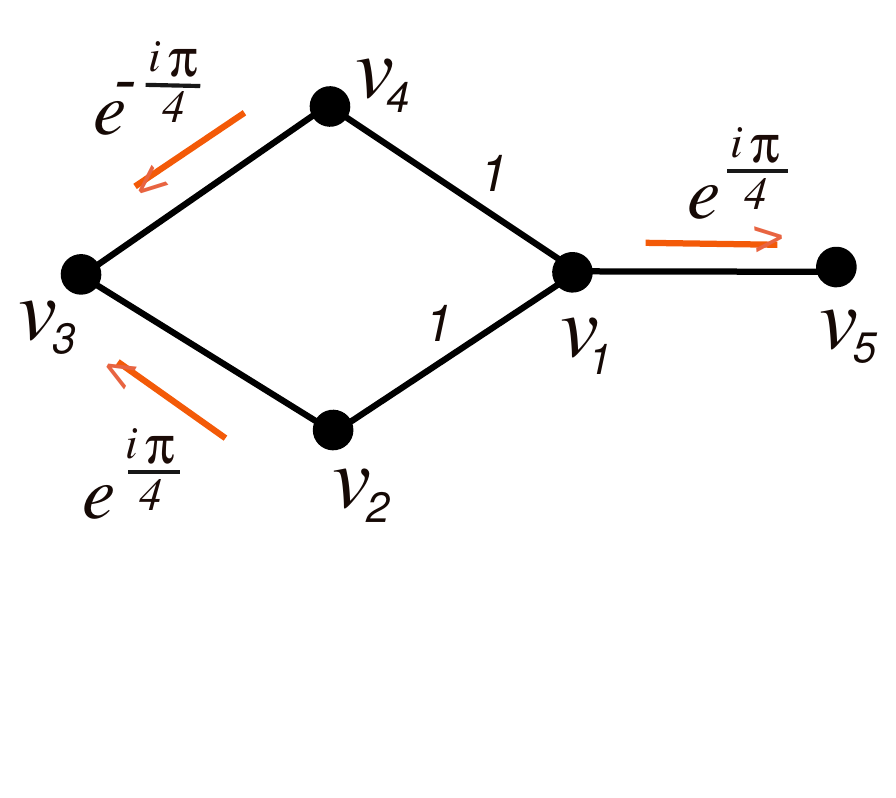}\\
  \caption{Complex unit gain graph $\Theta=(\Sigma,\varphi)$ with $V(\Sigma)=\{v_1, v_2, \cdots, v_5\}$.}\label{Fig1}
\end{figure}

\begin{definition}\label{def4a}
 (\textbf{Ordering independent} \cite{Samanta2022}).\\
 Let $\Theta=(\Sigma,\varphi)$ denote a complex unit gain graph with the standard vertex ordering $(V(\Sigma),<)$. Then $\Theta=(\Sigma,\varphi)$ is ordering independent (vertex ordering independent), if $D^{\max}_{<}(\Theta)=D^{\max}_{<_r}(\Theta)$ and $D^{\min}_{<}(\Theta)=D^{\min}_{<_r}(\Theta)$. If this is the case, then we define by $D^{\max}(\Theta):=D^{\max}_{<}(\Theta)=D^{\max}_{<_r}(\Theta)$ and $D^{\min}(\Theta):=D^{\min}_{<}(\Theta)=D^{\min}_{<_r}(\Theta)$.
\end{definition}

\begin{definition}\label{def4b}
 (\textbf{Distance compatible} \cite{Samanta2022}).\\
 Let $\Theta=(\Sigma,\varphi)$ denote a complex unit gain graph with the standard vertex ordering $(V(\Sigma),<)$. Then $\Theta=(\Sigma,\varphi)$ is said to be gain distance compatible or briefly compatible, if $D^{\max}_{<}(\Theta)=D^{\min}_{<}(\Theta)$. If this is the case, then we define by $D(\Theta):=D^{\max}_{<}(\Theta)=D^{\min}_{<}(\Theta)$.
\end{definition}

The two complete $\mathbb{T}$-gain graphs obtained by the two distance matrices $^{\max}_{<}(\Theta)$ and $^{\min}_{<}(\Theta)$ are given in the following definition:

\begin{definition}\label{def4c}
 \cite{Samanta2022}): Let $\Theta=(\Sigma,\varphi)$ denote a complex unit gain graph with the standard vertex ordering $(V(\Sigma),<)$. Then the complete graph corresponding to  $D^{\max}_{<}(\Theta)$ is denoted by $K^{D^{\max}_{<}}(\Theta)$ and defined as: for any $v_j,v_k \in V(\Sigma)$, we have $\varphi(v_j,v_k)=\varphi^{\max}_{<}(v_j,v_k)$. Similarly, the complete graph corresponding to  $D^{\min}_{<}(\Theta)$ is denoted by $K^{D^{\min}_{<}}(\Theta)$, for which we have $\varphi(v_j,v_k)=\varphi^{\min}_{<}(v_j,v_k)$.
\end{definition}
\ \

\subsection{Gain distance Laplacian matrices}

In \cite{Aouchiche2013}, Aouchiche and Hansen introduced the concept of distance Laplacian matrices for graphs. In this study, we generalize their results to the setting of complex unit gain graph and define two gain distance Laplacian matrices for a complex unit gain graph (or a $\mathbb{T}$-gain graph).\\
The transmission of a vertex $x$ is denoted by $tr(x)$, which is the sum of all the distances of $x$ to any other vertex of a graph. Equivalently, $tr(x)$ is defined as: $tr(x)=\sum_{x' \in V(\Sigma)}d(x,x')$. Let $\Sigma$ be a graph of order $n$, then the transmission matrix of $\Sigma$ is a $n \times n$ diagonal matrix, denoted by $Tr(\Sigma)$ and defined as:

$$Tr(\Sigma)=\left(
              \begin{array}{ccccc}
                tr(x_1) & 0 & \cdots & \cdots & 0\\
                0 & tr(x_2) & \cdots & \cdots & 0 \\
                \cdots & \cdots & \cdots & \cdots & \cdots \\
                \cdots & \cdots & \cdots & \cdots & \cdots \\
                0 & 0 & \cdots & \cdots & tr(x_n) \\
              \end{array}
            \right).
$$
\\

\begin{definition}\label{def4}
 (\textbf{Gain distance Laplacian matrices}).\\
 Let $\Theta=(\Sigma,\varphi)$ denote a connected complex unit gain graph with the vertex ordering $(V(\Sigma),<)$. Let $D^{\max}_{<}(\Theta)$ and $D^{\min}_{<}(\Theta)$ be the two gain distance matrices associated with a vertex ordering $(V(\Sigma),<)$ and $Tr(\Sigma)$ be the transmission matrix of $\Sigma$, then we define the two distance Laplacian matrices as follows: \\

(i) $DL^{\max}_{<}(\Theta)=Tr(\Sigma)-D^{\max}_{<}(\Theta)$;\\

(ii) $DL^{\min}_{<}(\Theta)=Tr(\Sigma)-D^{\min}_{<}(\Theta)$;\\

 Where, $\Theta$ is said to be distance compatible whenever, $DL(\Theta):=DL^{\max}_{<}(\Theta)=DL^{\min}_{<}(\Theta)$.\\

For $<_r$ the two distance Laplacian matrices are defined as follows: \\

(i) $DL^{\max}_{<_r}(\Theta)=Tr(\Sigma)-D^{\max}_{<_r}(\Theta)$;\\

(ii) $DL^{\min}_{<_r}(\Theta)=Tr(\Sigma)-D^{\min}_{<_r}(\Theta)$;\\
\end{definition}

The following example illustrate the above definitions.
\begin{example}
Let $\Theta=(\Sigma,\varphi)$ denote a connected complex unit gain graph with the standard vertex ordering $(V(\Sigma),<)$, as shown in Figure \ref{Fig1}. Then $DL^{\max}_{<}(\Theta)$ and $DL^{\min}_{<}(\Theta)$ for a graph $\Theta=(\Sigma,\varphi)$ with the given gains defined in Figure \ref{Fig1}, are presented in the following way, respectively.

$$DL^{\max}_{<}(\Theta)=\left(
              \begin{array}{ccccc}
                5 & -1 & -2e^{i\pi/4} & -1 & -e^{i\pi/4}\\
                -1 & 6 & -e^{i\pi/4} & -2 & -2e^{i\pi/4}\\
                -2e^{-i\pi/4} &- e^{-i\pi/4} & 7 & -e^{i\pi/4} & -3\\
                -1 & -2 & -e^{-i\pi/4} & 6 & -2e^{i\pi/4}\\
                -e^{-i\pi/4} & -2e^{-i\pi/4}& -3 & -2e^{-i\pi/4} & 8\\
              \end{array}
            \right).
$$
\\
For $<_r$, we have\\

$$DL^{\max}_{<_r}(\Theta)=\left(
              \begin{array}{ccccc}
                5 & -1 & -2e^{-i\pi/4} & -1 & -e^{i\pi/4}\\
                -1 & 6 & -e^{i\pi/4} & -2 & -2e^{i\pi/4}\\
                -2e^{i\pi/4} &-e^{-i\pi/4} & 7 & -e^{i\pi/4} & -3\\
                -1 & -2 & -e^{-i\pi/4} & 6 & -2e^{i\pi/4}\\
                -e^{-i\pi/4} & -2e^{-i\pi/4}& -3 & -2e^{-i\pi/4} & 8\\
              \end{array}
            \right).
$$
\\

It is easy to find all cases.\\
\end{example}

\section{Weighted $\mathbb{T}$-gain graphs}
In this section, we focus on the Laplacian matrix of weighted (positively weighted) $\mathbb{T}$-gain graphs. The Laplacian matrices of Weighted $\mathbb{T}$-gain graphs generalize notions for the Laplacian matrices of $\mathbb{T}$-gain graphs, Laplacian matrices of mixed graphs (where it is called the Hermitian Laplacian matrices), Laplacian matrices of signed graphs and that of undirected graphs. To observe balance in $\mathbb{T}$-gain graphs by studying $\mathbb{T}$-gain distance Lapalcain matrices, the Laplacian matrix of weighted $\mathbb{T}$-gain graphs is crucial. The main focus of this section is to study the three important results about those matrices, which are: the standard expression of the Laplacian matrix in term of incidence matrix; the expression for the Laplacian determinant; and the characterization of singularity. \\

\begin{definition}\label{def5}
 (\textbf{A weighted $\mathbb{T}$-gain graph}).\\
 A weighted or a positively weighted $\mathbb{T}$-gain graph $(\Theta,w)$ (or simply $\Theta_w$) is comprise by a $\mathbb{T}$-gain graph $\Theta=(\Sigma,\varphi)$, a positively weighted function $w$, (where, $w:E(\Sigma)\rightarrow \mathbb{R}^{+}$ and $E(\Sigma)$ is the undirected set of edges in $\Sigma$), and a weighted gain function $\varphi^w:\vec{E}(\Sigma)\rightarrow \mathbb{C}$, which is defined by:\\

 $$\varphi^w(\vec{e}_{j,k})=\varphi(\vec{e}_{j,k})w(e_{j,k}), ~~~ \text{for} ~\vec{e}_{j,k} \in \vec{E}(\Sigma),$$
 here, $\vec{e}_{j,k}$ represent an oriented edge from a vertex $v_j$ to a vertex $v_k$, and its opposite orientation is represented by $\vec{e}_{k,j}$. Note that if it is  an undirected edge from any vertex $v_j$ to $v_k$, then we denote it by simply $e_{j,k}$. \\

 The {\it adjacency matrix} of a weighted $\mathbb{T}$-gain graph is a square matrix with $|V(\Sigma)|=n$, defined by $A(\Theta,w)=(a_{jk})$, where

 $$a_{jk}=\left \{ \begin{array}{rcl}
\varphi^w(\vec{e}_{j,k}) & \mbox{if $e_{j,k} \in E(\Sigma)$}, \\
0 & \mbox{otherwise.}
\end{array} \right.$$

Since $A(\Theta,w)$ is Hermitian, therefore $A(\Theta,w)$ has real eigenvalues. Here, by the spectrum of $\Theta_w$, we mean the spectrum of $A(\Theta_w)$.
\end{definition}

\begin{definition}\label{def6}
 (\textbf{A weighted Lapalacian matrix $L(\Theta,w)$}).\\
 Let $(\Theta,w)$ be a weighted $\mathbb{T}$-gain graph, then the {\it weighted Laplacian matrix} is defined by $L(\Theta,w)=D(\Theta,w)-A(\Theta,w)$, where $D(\Theta,w)$ is the diagonal matrix defined by diag$(\Sigma_{e:v_j\sim e}w(e))$, and known as the {\it weighted degree matrix} of $(\Theta,w)$.
\end{definition}

Let $(\mathfrak{G},\varphi)$ be a gain graph, then the incidence matrix for $(\mathfrak{G},\varphi)$ is defined by Hameed and Ramakrishnan in \cite{Hameed2024}. Here we only take it for a particular choice where $(\mathfrak{G},\varphi)$ be a $\mathbb{T}$-gain graph.\\
Let $\vec{e}=\vec{v_jv_k}$ be the oriented edge, where the vertex $v_j$ considered to be the tail of $\vec{e}$, which is denoted by $\mathcal{\check{T}}(\vec{e})$, while the vertex $v_k$ is taken to be the head of $\vec{e}$, and denoted by $\mathcal{\hat{H}}(\vec{e})$.\\
 The (oriented) incidence matrix for a $\mathbb{T}$-gain graph $\Theta=(\Sigma,\varphi)$ is a $|V|\times |E|$ matrix defined by $H=H(\Theta)=h_{ve}$ and given by

 $$h_{ve}=\left \{ \begin{array}{rcl}
1 & \mbox{if $\mathcal{\check{T}}(\vec{e})=v$}, \\
-(\varphi(\vec{e}))^{-1} & \mbox{if $\mathcal{\hat{H}}(\vec{e})=v$}, \\
0 & \mbox{otherwise.}
\end{array} \right.$$

BY keeping the same orientation of edges as defined earlier. Next, we define the (oriented) incidence matrix of a weighted $\mathbb{T}$-gain graph. Since here we focus on the positively weighted graphs whose weighted function are positive real numbers so they always be appear in a square root form i.e., $\sqrt{w(e)}$.

\begin{definition}\label{def7}
 (\textbf{A weighted incidence matrix $H(\Theta,w)$}).\\
 Let $(\Theta,w)$ be a weighted $\mathbb{T}$-gain graph, then the (oriented) {\it weighted incidence matrix} of a $\mathbb{T}$-gain graph is defined by $H(\Theta,w)=(h_{ve})$ where

 $$h_{ve}=\left \{ \begin{array}{rcl}
\sqrt{w(e)} & \mbox{if $\mathcal{\check{T}}(\vec{e})=v$}, \\
-(\varphi(\vec{e}))^{-1}\sqrt{w(e)} & \mbox{if $\mathcal{\hat{H}}(\vec{e})=v$}, \\
0 & \mbox{otherwise.}
\end{array} \right.$$
\end{definition}

Let denote the transpose of $H(\Theta,w)=h_{ve}$ by $H^T(\Theta,w)=h'_{ev}$, then $h'_{ev}=h_{ve}$. Next, we provide the generalization of the standard formula which connect the Lapalacian and incidence matrices.

\begin{theorem}\label{Theorem1}
Let $(\Theta,w)$ be a weighted $\mathbb{T}$-gain graph, then $L(\Theta,w)=H(\Theta,w)H^T(\Theta,w)$.
\end{theorem}

\begin{proof}
Let $V(\Theta)=\{ v_1,v_2, \cdots, v_n\}$ and $\vec{E}(\Theta)=\{ \vec{e}_1, \vec{e}_2, \cdots, \vec{e}_m \}$ be the vertex and edge sets of $(\Theta,w)$, respectively. Let denoting $H(\Theta,w)$ by $(h_{v_j\vec{e}_k})$ and $H^T(\Theta,w)$ by $(h'_{\vec{e}_jv_k})$. Let the $j^{th}$ row vector of $H(\Theta,w)$ be $(h_{v_j\vec{e}_1},h_{v_j\vec{e}_2}, \cdots, h_{v_j\vec{e}_m})$ and the $k^{th}$ column of $H^T(\Theta,w)$ be $(h'_{\vec{e}_1v_k},h'_{\vec{e}_2v_k}, \cdots, h'_{\vec{e}_mv_k})^T$.\\

Now the $(j,k)^{th}$ entry of $HH^T$ is $\sum_{i=1}^{m}h_{v_j\vec{e}_i}h'_{\vec{e}_iv_k}$.\\

For $j=k$, $h_{v_j\vec{e}_i}h'_{\vec{e}_iv_k} \neq 0$ if and only if $\vec{e}_i$ is incident to $v_j$. Now, if $\mathcal{\check{T}}(\vec{e}_i)=v_j$, then $h_{v_j\vec{e}_i}=\sqrt{w(\vec{e}_i)}$ in $H(\Theta,w)$ and for $H^T(\Theta,w)$, we have the similar expression such as $h'_{\vec{e}_iv_k}=\sqrt{w(\vec{e}_i)}$. Thus, $h_{v_j\vec{e}_i}h'_{\vec{e}_iv_k}=\pm (\sqrt{w(\vec{e}_i)})^2=w(e_i)$ and hence the diagonal entry of $H(\Theta,w)H^T(\Theta,w)$ is equal to $\sum_{e:v_j \sim e}w(e)$.\\

For $j \neq k$, $h_{v_j\vec{e}_i}h'_{\vec{e}_iv_k} \neq 0$ if and only if an edge $\vec{e}_i$ joining $v_j$ and $v_k$. If $\vec{e}_i=\overrightarrow{v_jv_k}$, then

 $$h_{v_j\vec{e}_i}h'_{\vec{e}_iv_k}=\sqrt{w(\vec{e_i})}(-\sqrt{w(\vec{e_i})}((\varphi(e_i))^{-1})^{-1})=-w(e_i) \varphi(e_i).$$

If $\vec{e}_i=\overrightarrow{v_kv_j}$, then we have

$$h_{v_j\vec{e}_i}h'_{\vec{e}_iv_k}=\sqrt{w(\vec{e_i})}(-\sqrt{w(\vec{e_i})}(\varphi(e_i))^{-1})=-w(e_i)(\varphi(e_i))^{-1}.$$

Clearly, for all possible cases the $(j,k)^{th}$ entry in $H(\Theta,w)H^T(\Theta,w)$ coincides with the $(j,k)^{th}$ entry of $L(\Theta,w)$, hence this complete the proof.
\end{proof}

\begin{lemma}\label{Lemma1}
Let $(\Theta,w)$ be a weighted $\mathbb{T}$-gain tree, then $\mathrm{det}(L(\Theta,w))=0$.
\end{lemma}

\begin{proof}
Let $(\Theta,w)$ be a weighted $\mathbb{T}$-gain tree of order $n$. Then, $(\Theta,w)$ consist of $n-1$ edges and thus $H(\Theta,w)$ is a matrix with order $n \times (n-1)$. Since $L(\Theta,w)=H(\Theta,w)H^T(\Theta,w)$. Therefore, the rank of $L(\Theta,w)$ is less than $n$, which implies $\det(L(\Theta,w))=0$.
\end{proof}

\begin{theorem}\label{Theorem2}
Let $(\Theta,w)$ be a weighted $\mathbb{T}$-gain cycle $\overrightarrow{C_n}$ of order $n$, then

 $$\det(L(\Theta,w))=\prod_{\vec{e}\in E(\overrightarrow{C_n})}w(\vec{e})\bigg(2(1-Re(\varphi(\overrightarrow{C_n}))\bigg),$$
 where, $Re(\varphi(\overrightarrow{C_n}))$ is the real part of $\varphi(\overrightarrow{C_n})$.
\end{theorem}

\begin{proof}
Let $(\Theta,w)$ be a weighted $\mathbb{T}$-gain cycle $\overrightarrow{C_n}$, in which the vertices and edges are sequenced by: $\overrightarrow{C_n}=v_1\vec{e_1}v_2\vec{e_2} \cdots v_{n-1}\vec{e}_{n-1}v_n\vec{e}_n$. Now, the weighted $\mathbb{T}$-gain incidence matrix $H(\Theta,w)$ for $\overrightarrow{C_n}$ is given by:

$$\left(
              \begin{array}{cccccc}
                \sqrt{w(\vec{e}_1)} & 0 & 0 & \cdots & 0 & -\sqrt{w(\vec{e}_n)} (\varphi(\vec{e}_n))^{-1}\\
                -\sqrt{w(\vec{e}_1)}(\varphi(\vec{e}_1))^{-1} & \sqrt{w(\vec{e}_2)} & 0 & \cdots & 0 & 0 \\
                0 & -\sqrt{w(\vec{e}_2)}(\varphi(\vec{e}_2))^{-1} & \sqrt{w(\vec{e}_3)} & \cdots & 0 & 0 \\
                \vdots & \vdots & \vdots & \ddots  & \vdots & \vdots  \\
                0 & 0 & 0 & \cdots & \sqrt{w(\vec{e}_{n-1})} & 0 \\
                0 & 0 & 0 & \cdots & -\sqrt{w(\vec{e}_{n-1})}(\varphi(\vec{e}_{n-1}))^{-1} & \sqrt{w(\vec{e}_n)} \\
              \end{array}
            \right).
$$
\\
To find the determinant of the above square matrix, we expand it by the first row. Thus, we get the following expression for the determinant of $H(\Theta,w)$.\\
$$\det(H(\Theta,w))=\sqrt{w(\vec{e}_1)} M_{1,1}+(-1)^n \sqrt{w(\vec{e}_n)} (\varphi(\vec{e}_n))^{-1}M_{1,n},$$

where,

$$M_{1,1}=\det \left(
              \begin{array}{ccccc}
                \sqrt{w(\vec{e}_2)} & 0 & \cdots & 0 & 0 \\
               -\sqrt{w(\vec{e}_2)}(\varphi(\vec{e}_2))^{-1} & \sqrt{w(\vec{e}_3)} & \cdots & 0 & 0 \\
                 \vdots & \vdots & \ddots  & \vdots & \vdots  \\
                 0 & 0 & \cdots & \sqrt{w(\vec{e}_{n-1})} & 0 \\
                 0 & 0 & \cdots & -\sqrt{w(\vec{e}_{n-1})}(\varphi(\vec{e}_{n-1}))^{-1} & \sqrt{w(\vec{e}_n)} \\
              \end{array}
            \right),
$$
\\
and

$$M_{1,n}=\det \left(
              \begin{array}{ccccc}
                -\sqrt{w(\vec{e}_1)}(\varphi(\vec{e}_1))^{-1} & \sqrt{w(\vec{e}_2)} & 0 & \cdots & 0  \\
                0 & -\sqrt{w(\vec{e}_2)}(\varphi(\vec{e}_2))^{-1} & \sqrt{w(\vec{e}_3)} & \cdots & 0  \\
                \vdots & \vdots & \vdots & \ddots  & \vdots  \\
                0 & 0 & 0 & \cdots & \sqrt{w(\vec{e}_{n-1})} \\
                0 & 0 & 0 & \cdots & -\sqrt{w(\vec{e}_{n-1})}(\varphi(\vec{e}_{n-1}))^{-1} \\
              \end{array}
            \right).
$$
\\
Since, both $M_{1,1}$ and $M_{1,n}$ are triangular matrices, and hence the determinants for such types of matrices are always equal to the product of their diagonal entries. So that

$$M_{1,1}=\sqrt{w(\vec{e}_2)}\sqrt{w(\vec{e}_3)} \sqrt{w(\vec{e}_4)} \cdots \sqrt{w(\vec{e}_n)},$$
 and
$$M_{1,n}=(-\sqrt{w(\vec{e}_1)}(\varphi(\vec{e}_1))^{-1})(-\sqrt{w(\vec{e}_2)}(\varphi(\vec{e}_2))^{-1}) \cdots (-\sqrt{w(\vec{e}_{n-1})}(\varphi(\vec{e}_{n-1}))^{-1}).$$
\\
Hence,
\begin{eqnarray*}
  \det(H(\Theta,w)) &=& \sqrt{w(\vec{e}_1)}\sqrt{w(\vec{e}_2)}\sqrt{w(\vec{e}_3)}\cdots \sqrt{w(\vec{e}_n)}
 +((-1)^n\sqrt{w(\vec{e}_n)}(\varphi(\vec{e}_n))^{-1}) \\
   & & (-\sqrt{w(\vec{e}_1)}(\varphi(\vec{e}_1))^{-1})
  \cdots (-\sqrt{w(\vec{e}_{n-1})}(\varphi(\vec{e}_{n-1}))^{-1}), \\
  \\
   &=&\prod_{\vec{e}\in E(\overrightarrow{C_n})}\sqrt{w(\vec{e})}\bigg(1-(\varphi(\overrightarrow{C_n}))^{-1}\bigg).
\end{eqnarray*}

Next, we have to find $H^T(\Theta,w)$ for a weighted $\mathbb{T}$-gain cycle $\overrightarrow{C_n}$, which is given by:

$$\left(
              \begin{array}{cccccc}
                \sqrt{w(\vec{e}_1)} & -\sqrt{w(\vec{e}_1)}\varphi(\vec{e}_1)) & 0 & \cdots & 0 & 0\\
                0 & \sqrt{w(\vec{e}_2)} & -\sqrt{w(\vec{e}_2)}\varphi(\vec{e}_2)) & \cdots & 0 & 0 \\
                0 & 0 & \sqrt{w(\vec{e}_3)} & \cdots & 0 & 0 \\
                \vdots & \vdots & \vdots & \ddots  & \vdots & \vdots  \\
                0 & 0 & 0 & \cdots & \sqrt{w(\vec{e}_{n-1})} & -\sqrt{w(\vec{e}_{n-1})}\varphi(\vec{e}_{n-1}) \\
                -\sqrt{w(\vec{e}_n)} \varphi(\vec{e}_n)) & 0 & 0 & \cdots & 0 & \sqrt{w(\vec{e}_n)} \\
              \end{array}
            \right).
$$
\\
To find the determinant of the above square matrix, we expand it by the first column. Thus, we get the following expression for the determinant of $H^T(\Theta,w)$.\\
$$\det(H^T(\Theta,w))=\sqrt{w(\vec{e}_1)} M_{1,1}+(-1)^n\sqrt{w(\vec{e}_1)} \varphi(\vec{e}_1)M_{n,1},$$

where,

$$M_{1,1}=\det \left(
              \begin{array}{ccccc}
                \sqrt{w(\vec{e}_2)} & -\sqrt{w(\vec{e}_2)}\varphi(\vec{e}_2) & \cdots & 0 & 0 \\
                 0 & \sqrt{w(\vec{e}_3)} & \cdots & 0 & 0 \\
                 \vdots & \vdots & \ddots  & \vdots & \vdots  \\
                 0 & 0 & \cdots & \sqrt{w(\vec{e}_{n-1})} & -\sqrt{w(\vec{e}_{n-1})}\varphi(\vec{e}_{n-1}) \\
                 0 & 0 & \cdots & 0 & \sqrt{w(\vec{e}_n)} \\
              \end{array}
            \right),
$$
\\
and

$$M_{n,1}=\left(
              \begin{array}{ccccc}
                 -\sqrt{w(\vec{e}_1)}\varphi(\vec{e}_1)) & 0 & \cdots & 0 & 0\\
                \sqrt{w(\vec{e}_2)} & -\sqrt{w(\vec{e}_2)}\varphi(\vec{e}_2)) & \cdots & 0 & 0 \\
                0 & \sqrt{w(\vec{e}_3)} & \cdots & 0 & 0 \\
                 \vdots & \vdots & \ddots  & \vdots & \vdots  \\
                 0 & 0 & \cdots & \sqrt{w(\vec{e}_{n-1})} & -\sqrt{w(\vec{e}_{n-1})}\varphi(\vec{e}_{n-1}) \\
              \end{array}
            \right).
$$
\\
Since, both $M_{1,1}$ and $M_{n,1}$ are triangular matrices, and hence the determinants for such types of matrices are always equal to the product of their diagonal entries. So that

$$M_{1,1}=\sqrt{w(\vec{e}_2)}\sqrt{w(\vec{e}_3)} \sqrt{w(\vec{e}_4)} \cdots \sqrt{w(\vec{e}_n)},$$
 and
$$M_{n,1}=(-\sqrt{w(\vec{e}_1)}\varphi(\vec{e}_1))(-\sqrt{w(\vec{e}_2)}\varphi(\vec{e}_2)) \cdots (-\sqrt{w(\vec{e}_{n-1})}\varphi(\vec{e}_{n-1})).$$
\\
Hence,
\begin{eqnarray*}
  \det(H^T(\Theta,w)) &=& \sqrt{w(\vec{e}_1)}\sqrt{w(\vec{e}_2)}\sqrt{w(\vec{e}_3)}\cdots \sqrt{w(\vec{e}_n)}
 +((-1)^n\sqrt{w(\vec{e}_n)}\varphi(\vec{e}_n)) \\
   & & (-\sqrt{w(\vec{e}_1)}\varphi(\vec{e}_1))
  \cdots (-\sqrt{w(\vec{e}_{n-1})}\varphi(\vec{e}_{n-1})), \\
  \\
   &=&\prod_{\vec{e}\in E(\overrightarrow{C_n})}\sqrt{w(\vec{e})}\bigg(1-\varphi(\overrightarrow{C_n})\bigg).
\end{eqnarray*}
Since, $\det(L(\Theta,w))=\det(H(\Theta,w)).\det(H^T(\Theta,w))$. This, implies that

\begin{eqnarray*}
  \det(L(\Theta,w)) &=& \prod_{\vec{e}\in E(\overrightarrow{C_n})}\sqrt{w(\vec{e})}\bigg(1-(\varphi(\overrightarrow{C_n}))^{-1}\bigg).\prod_{\vec{e}\in E(\overrightarrow{C_n})}\sqrt{w(\vec{e})}\bigg(1-\varphi(\overrightarrow{C_n})\bigg) \\
  \\
   &=& \prod_{\vec{e}\in E(\overrightarrow{C_n})}w(\vec{e})\bigg(2(1-Re(\varphi(\overrightarrow{C_n}))\bigg).
\end{eqnarray*}

\end{proof}

The following corollary is the immediate consequence of Theorem \ref{Theorem2}.
\begin{corollary}\label{corollary1}
Let $(\Theta,w)$ be a weighted $\mathbb{T}$-gain cycle $\overrightarrow{C_n}$ of order $n$, then $\det(L(\Theta,w))=0$, if and only if $\overrightarrow{C_n}$ is balanced, i.e., $\varphi(\overrightarrow{C_n})=(\varphi(\overrightarrow{C_n}))^{-1}=1$.
\end{corollary}

A graph is said to be $1${\it-tree} if it is connected and also unicyclic. Now, we define a graph to be {$1${\it-forest} if all of its components are $1$-trees or more precisely we say that a $1$-forest is a disjoint union of $1$-trees. A graph is said to be {\it spanning $1$-forest} of a graph $\Sigma$, if it is a $1$-forest and also a spanning subgraph of graph $\Sigma$. The collection of all spanning $1$-forest of $\Sigma$ is denoted by $\Psi(\Sigma)$.

\begin{lemma}\label{Lemma2}
Let $(\Theta,w)=(\Sigma, \varphi, w)$ be a weighted $\mathbb{T}$-gain graph, where the underlying graph $\Sigma=(V,E)$ is a $1$-tree of order $n$, with a unique cycle $\overrightarrow{C_{q}}$ of order $q$. Then

 $$\det(L(\Theta,w))=\prod_{\vec{e}\in \vec{E}(\Theta,w)}w(\vec{e})\bigg(2(1-Re(\varphi(\overrightarrow{C_q}))\bigg).$$
\end{lemma}

\begin{proof}
Let $\overrightarrow{C_q}$ be the unique cycle. Let sequenced the vertices and edges in $\overrightarrow{C_q}$ by the following pattern:\\ $$\overrightarrow{C_q}=v_1\vec{e_1}v_2\vec{e_2} \cdots v_{q-1}\vec{e}_{q-1}v_q\vec{e}_q.$$
Now, we define the orientation of those edges which are not belonging to $\overrightarrow{C_q}$. Consider for $j<k$, the edge $\vec{e}_{j,k}$ has tail $\mathcal{T}(\vec{e}_{j,k})=v_j$ and head $\mathcal{H}(\vec{e}_{j,k})=v_k$. Let labelled the vertices in such a way that every edge $\vec{e}_{j,k}$ not in $\overrightarrow{C_q}$ has $v_j$ nearer to $\overrightarrow{C_q}$ than $v_k$. Then, the weighted $\mathbb{T}$-gain incidence matrix $H(\Theta,w)$ is given by:

$$\left(
              \begin{array}{c|ccccc}
              H(\overrightarrow{C_q},w) &  &  & \#  &  & \\
              \hline
                 & \sqrt{w(\vec{e}_{q+1})}& \# & \cdots & \# & \#\\
                 & 0 & \sqrt{w(\vec{e}_{q+2})} & \cdots & \# & \# \\
                O & 0 & 0 & \cdots & \# & \# \\
                 & \vdots & \vdots & \ddots  & \vdots & \vdots  \\
                 & 0 & 0 & \cdots & 0 & \sqrt{w(\vec{e}_{n})} \\
              \end{array}
            \right).
$$
\\
Clearly, it is a triangular (upper-triangular) block matrix with the first diagonal block is $H(\overrightarrow{C_q},w)$; whereas the other main diagonals entries are corresponding to the heads of those edges which are not belonging to $\overrightarrow{C_q}$. Thus, we get the following expression for the determinant of $H(\Theta,w)$.\\

\begin{eqnarray*}
  \det(H(\Theta,w)) &=& \det(H(\overrightarrow{C_q},w)). \prod_{p=q+1}^{n}\sqrt{w(\vec{e_p})} \\
   &=& \prod_{p=1}^{q}\sqrt{w(\vec{e_p})}\bigg(1-(\varphi(\overrightarrow{C_q}))^{-1}\bigg).\prod_{p=q+1}^{n}\sqrt{w(\vec{e_p})}\\
   &=&\prod_{\vec{e}\in \vec{E}(\Theta,w)}\sqrt{w(\vec{e})}\bigg(1-(\varphi(\overrightarrow{C_q}))^{-1}\bigg).
\end{eqnarray*}

It is easy to find it for $\det(H^T(\Theta,w))$. Here, we can write it as follows:\\

$$\det(H^T(\Theta,w))=\prod_{\vec{e}\in \vec{E}(\Theta,w)}\sqrt{w(\vec{e})}\bigg(1-\varphi(\overrightarrow{C_q})\bigg).$$

Since, $\det(L(\Theta,w))=\det(H(\Theta,w)).\det(H^T(\Theta,w))$. This, implies that

\begin{eqnarray*}
  \det(L(\Theta,w)) &=& \prod_{\vec{e}\in \vec{E}(\Theta,w)}\sqrt{w(\vec{e})}\bigg(1-(\varphi(\overrightarrow{C_q}))^{-1}\bigg).\prod_{\vec{e}\in \vec{E}(\Theta,w)}\sqrt{w(\vec{e})} \bigg(1-\varphi(\overrightarrow{C_q})\bigg) \\
  \\
   &=& \prod_{\vec{e}\in \vec{E}(\Theta,w)}w(\vec{e})\bigg(2(1-Re(\varphi(\overrightarrow{C_q}))\bigg).
\end{eqnarray*}
\end{proof}

\begin{lemma}\label{Lemma3}
Let $(\Theta,w)=(\Sigma, \varphi, w)$ be a weighted $\mathbb{T}$-gain graph, where the underlying graph $\Sigma=(V,E)$ is a $1$-forest of order $n$, then

 $$\det(L(\Theta,w))=\prod_{\vec{e}\in \vec{E}(\Theta,w)}w(\vec{e})~{\prod}_\Upsilon \bigg(2(1-Re(\varphi(\overrightarrow{C_\Upsilon}))\bigg),$$
 where the second product is taken over all component $1$-trees $\Upsilon$ of $(\Theta,w)$, with an unique cycle $\overrightarrow{C_\Upsilon}$.
\end{lemma}

\begin{proof}
Let $(\Theta,w)$ be a $1$-forest with $p$-components. Let there are $H_1, H_2, \cdots, H_p$ incidence matrices corresponding to the components of $(\Theta,w)$. Let denote the direct sum of matrices by symbol $\bigoplus$, which define the block diagonal matrix of any matrices. By arranging the vertices and edges of $(\Theta,w)$ in a suitable way to make the incidence matrix $H(\Theta,w)$ into a block diagonal matrix $\bigoplus_{j=1}^{p}H_j$, where each $H_j$ is corresponding to the incidence matrix of a $1$-tree component of the $1$-forest and clearly, $p$ stand for the total components of $1$-forest. Then, $\det(H(\Theta,w))=\prod_{j=1}^{p}\det(H_j)$. The rest of the proof to be completed in the light of the above lemma, see Lemma \ref{Lemma2}.
\end{proof}

\begin{lemma}\label{Lemma4}
Let $(\Theta,w)$ be a weighted $\mathbb{T}$-gain graph of order $n$ and let $\Upsilon$ be a spanning subgraph of $(\Theta,w)$ consist of exactly $n$ number of edges. Then $\det(L(\Upsilon)) \neq 0$ if and only if $\Upsilon$ is a spanning $1$-forest of $(\Theta,w)$.
\end{lemma}

\begin{proof}
Let a spanning subgraph $\Upsilon$ of $(\Theta,w)$ consist of exactly $n=V(\Theta,w)$ edges. Let $\det(L(\Upsilon)) \neq 0$, then to prove $\Upsilon$ is a spanning $1$-forest of $(\Theta,w)$, we need to prove the components of $\Upsilon$ are $1$-trees. For this reason, arranging the vertices and edges of $\Upsilon$ in a suitable way to make the Laplacian matrix $L(\Upsilon)$ into a block diagonal matrix $\bigoplus_{j=1}^{p}L_j$, where all $L_j$ are corresponding to the components of $\Upsilon$, and the term $p$ represent the number of components in $\Upsilon$. This implies that,

$$\det(L(\Upsilon))=\prod_{j=1}^{p}\det(L_j).$$

Assume that $\Upsilon$ has an isolated vertex. In this case, $L(\Upsilon)$ contains a row which is zero. This implies that $\det(L(\Upsilon))=0$. Now consider the case, when some component of $\Upsilon$ is a tree for some $j$, by Lemma \ref{Lemma1}, we have $\det(L_j)=0$, which further implies that $\det(L(\Upsilon))=0$. Thus, in both cases, the contradiction arise.\\

\textbf{Claim:} If $B_p$ is a component of $\Upsilon$ and all the $B_p$,s consist of the same number of vertices and edges.\\
 Assume $B_p$, for some $p$, has $q$ vertices and $q+k$ edges, where $k \geq 1$. Then the $n-q$ vertices and $n-q-k$ edges not in $B_p$ form either a tree or a disconnected graph containing trees as components. Thus, in both cases $\det(L(\Upsilon))=0$, which is contradiction. Thus, our claim.\\
 So that all the components of $\Upsilon$ consist of the same number of vertices and edges implies that all the components $B_p$ of $\Upsilon$ are $1$-trees and thus $\Upsilon$ is a spanning $1$-forest. Hence, $\Upsilon$ is a spanning $1$-forest of $(\Theta,w)$.\\

 The converse part can be proved by using Kirchhoff's Matrix-Tree-Theorem with a little more calculations.

\end{proof}

\begin{theorem}\label{Theroem3}
Let $(\Theta,w)=(\Sigma, \varphi, w)$ be a connected weighted $\mathbb{T}$-gain graph, then

 $$\det(L(\Theta,w))=\sum_{\Upsilon \in \Psi(\Sigma)}\prod_{\vec{e}\in \Upsilon}w(\vec{e})~{\prod}_{q \in \Upsilon}\bigg(2(1-Re(\varphi(\overrightarrow{C_q}))\bigg),$$
 where the summation taken over all spanning $1$-forest $\Upsilon$ of $(\Theta,w)$ and $q \in \Upsilon$ is the $1$-trees component $q$ in the spanning $1$-forest.
\end{theorem}

\begin{proof}
Since $L(\Theta,w)=H(\Theta,w)H^T(\Theta,w)$, by the Binet-Cauchy theorem \cite{Broida1989}, we obtain
\begin{equation*}
\text{det}(L(\Theta,w))=\sum_{M}\text{det}(H(M))\text{det}H^T(M)=\sum_{M}\text{det} (L(M)),
\end{equation*}

where $M$ is a spanning subgraph of $\Sigma$ with exactly $n$ number of edges.\\

Next, Lemma \ref{Lemma4} implies that
$$\text{det}(L(\Theta,w))=\sum_{\Upsilon \in \Psi(\Sigma)}\text{det}(L(\Upsilon),$$
where the summation taken over all spanning $1$-forests $\Upsilon$ of $\Sigma$.\\
Hence, Lemma \ref{Lemma3} complete the proof of the theorem.

\end{proof}

Let denote the rank of a matrix $M$ by $\text{rank}(M)$, which is the order of the largest submatrix (square matrix) of $M$ of non-zero determinant. The following theorems characterize balance in weighted $\mathbb{T}$-gain graphs.

\begin{theorem}\label{Theorem4}
Let $(\Theta,w)=(\Sigma, \varphi, w)$ be a connected unbalanced weighted $\mathbb{T}$-gain graph of order $n$, then $\text{rank}(H(\Theta,w))=n$ and $\text{rank}(H^T(\Theta,w))=n$.
\end{theorem}

\begin{proof}
Consider an unbalanced cycle $\overrightarrow{C}=v_1\vec{e}_1v_2\vec{e}_2 \cdots \vec{e}_{l-1}v_l\vec{e}_lv_1$ in $(\Theta,w)$ with length $l$, where $3 \leq l \leq n$ in $(\Theta,w)$. Since $\overrightarrow{C}$ is unbalanced, so that $\varphi(\overrightarrow{C}) \neq 1$. Choosing a subgraph $\Upsilon$, which is a spanning $1$-tree in $(\Theta,w)$, contains a unique cycle $\overrightarrow{C}$. Since $(\Theta,w)$ connected, so such type of selection is possible. Let's arrange the vertices and edges in a suitable way and let denote the remaining edges in $\Upsilon$ by $\vec{e}_{l+1},\vec{e}_{l+2}, \cdots , \vec{e}_{n}$, so that for $j<k$, the edge $\vec{e}=\overrightarrow{v_jv_k}$ has tail $\mathcal{T}(\vec{e})=v_j$ and head $\mathcal{H}(\vec{e})=v_k$.\\
Since $H(\Upsilon)$ is a $n \times n$ submatrix, obtained by $\Upsilon$, where the edges of $\Upsilon$ are corresponding to the columns in $H(\Upsilon)$. Then the determinant of $H(\Upsilon)$ is given by

\begin{eqnarray*}
  \det(H(\Upsilon) &=& \det(H(\overrightarrow{C},w)). \prod_{p=l+1}^{n}\sqrt{w(\vec{e_p})} .
\end{eqnarray*}

Clearly, $\prod_{p=l+1}^{n}\sqrt{w(\vec{e_p})} \neq 0$. Since weights are always positive real numbers and can't be zero. In the other hand, $\det(H(\overrightarrow{C},w))=\sqrt{w(\overrightarrow{C})}\bigg(1-(\varphi(\overrightarrow{C})^{-1}\bigg)$. The only possibility for the determinant of  $H(\overrightarrow{C},w)$ to be zero, when $(\varphi(\overrightarrow{C}))^{-1}=1$. Since $\overrightarrow{C}$ is unbalanced, so that $(\varphi(\overrightarrow{C}))^{-1} \neq 1$. Thus, $\det(H(\Upsilon) \neq 0$, and hence $\mathrm{rank}(H(\Upsilon))=n$, which implies $\mathrm{rank}(H(\Theta,w)=n$.\\

By the similar way, we can follow for $H^T(\Theta,w)$.

\begin{eqnarray*}
  \det(H^T(\Theta,w) &=& \sqrt{w(\overrightarrow{C})}\bigg(1-\varphi(\overrightarrow{C}\bigg). \prod_{p=l+1}^{n}\sqrt{w(\vec{e_p})} \neq 0
\end{eqnarray*}
this implies, $\mathrm{rank}(H^T(\Theta,w)=n$.
\end{proof}

\begin{theorem}\label{Theorem5}
Let $(\Theta,w)=(\Sigma, \varphi, w)$ be a connected balanced weighted $\mathbb{T}$-gain graph of order $n$, then $\text{rank}(H(\Theta,w))=n-1$ and $\mathrm{rank}(H^T(\Theta,w))=n-1$.
\end{theorem}

\begin{proof}
Since $(\Theta,w)$ is connected, so it is possible to consider a spanning tree in $(\Theta,w)$ of order $n-1$, and then follow the same proof technique adopted in the above theorem.
\end{proof}

\begin{theorem}\label{Theorem6}
Let $(\Theta,w)=(\Sigma, \varphi, w)$ be a connected weighted $\mathbb{T}$-gain graph, then $(\Theta,w)$ is balanced if and only if $\det(L(\Theta,w))=0$.
\end{theorem}

\begin{proof}
If $(\Theta,w)$ is balanced, then all the of cycles in $(\Theta,w)$ must satisfy the condition $2-[\varphi(\overrightarrow{C})+(\varphi(\overrightarrow{C}))^{-1}]=0$. Thus, by matrix tree theorem we get $\det(L(\Theta,w))=0$. Conversely, assume on contrary. Let $\det(L(\Theta,w)) \neq 0$. This implies $\mathrm{rank}(L(\Theta,w))=n$ and thus $\mathrm{rank}(H(\Theta,w))=n$, which give us $(\Theta,w)$ is unbalanced.
\end{proof}

\section{$\mathbb{T}$-gain distance Laplacians}
We start this section with the definitions of $\mathbb{T}$-gain distance Laplacian for a connected $\mathbb{T}$-gain graph in term of a $\mathbb{T}$-gain distance incidence matrix according to the Definition \ref{def7}. Since the underlying graph to be considered the complete graph $K_n$, so for every edge $v_jv_k \in E(K_n)$ we a have column in the incidence matrix, while the weight of edges represent the distances such as: $w(v_jv_k)=d(v_j,v_k)$.  Then for arbitrary oriented edges we have the following definition.

\begin{definition}\label{def8}
 (\textbf{A maximum distance incidence matrix $DH^{\max}_{<}(\Theta)$}).\\
 Let $\Theta=(\Sigma,\varphi)$ be a $\mathbb{T}$-gain graph, then the (oriented) {\it maximum distance incidence matrix} of a $\mathbb{T}$-gain graph is defined by $DH^{\max}_{<}(\Theta)=(h_{v_j,v_jv_k})$ where

 $$h_{v_j,v_jv_k}=\left \{ \begin{array}{rcl}
\sqrt{d(v_j,v_k)} & \mbox{if $\mathcal{\check{T}}(v_jv_k)=v_j$}, \\
-(\varphi^{\max}_{<}(v_jv_k))^{-1}\sqrt{d(v_j,v_k)} & \mbox{if $\mathcal{\hat{H}}(v_jv_k)=v_j$}, \\
0 & \mbox{otherwise.}
\end{array} \right.$$
\end{definition}

\begin{definition}\label{def9}
 (\textbf{A minimum distance incidence matrix $DH^{\min}_{<}(\Theta)$}).\\
 Let $\Theta=(\Sigma,\varphi)$ be a $\mathbb{T}$-gain graph, then the (oriented) {\it minimum distance incidence matrix} of a $\mathbb{T}$-gain graph is defined by $DH^{\min}_{<}(\Theta)=(h_{v_j,v_jv_k})$ where

 $$h_{v_j,v_jv_k}=\left \{ \begin{array}{rcl}
\sqrt{d(v_j,v_k)} & \mbox{if $\mathcal{\check{T}}(v_jv_k)=v_j$}, \\
-(\varphi^{\min}_{<}(v_jv_k))^{-1}\sqrt{d(v_j,v_k)} & \mbox{if $\mathcal{\hat{H}}(v_jv_k)=v_j$}, \\
0 & \mbox{otherwise.}
\end{array} \right.$$
\end{definition}

Next, the Theorem \ref{Theorem1} gives the formulas for $\mathbb{T}$-gain distance Laplacians.

\begin{theorem}\label{Theorem7}
For a connected $\mathbb{T}$-gain graph $\Theta=(\Sigma,\varphi)$,
\begin{eqnarray*}
  DL^{\max}_{<} &=& DH^{\max}_{<} (DH^{\max}_{<})^{T},\\
  DL^{\min}_{<} &=& DH^{\min}_{<} (DH^{\min}_{<})^{T},\\
  DL(\Theta)    &=&  DH(\Theta) (DH(\Theta))^T.
\end{eqnarray*}
\end{theorem}

Next, we find expressions for the determinants, which are based on Theorem \ref{Theroem3}. Let $\Upsilon$ be the subgraph of $K^{D^{\max}_<}$, then we have the following expression for the determinant.

\begin{theorem}\label{Theroem8}
Let $\Theta=(\Sigma, \varphi)$ be a connected $\mathbb{T}$-gain graph, then

 $$\det(DL^{\max}_{<}(\Theta)=\sum_{\Upsilon \in K^{D^{\max}_<}}\prod_{\vec{e}\in \Upsilon}w(\vec{e})~{\prod}_{\overrightarrow{C} \in \Upsilon}\bigg(2(1-Re(\varphi(\overrightarrow{C}))\bigg),$$
 wit the similar formulas for $\det(DL^{\min}_{<}(\Theta))$ and $\det(DL(\Theta))$, where the summation taken over all essential spanning $1$-forest $\Upsilon$ of $K^{D^{\max}_<}(\Theta)$, $K^{D^{\min}_<}(\Theta)$ and $K^{D}(\Theta)$, respectively.
\end{theorem}

The following is the ordering independent theorem proved by Samanta and Kanan in \cite{Samanta2022}.
\begin{theorem}\label{Theroem9}
\cite{Samanta2022} Let $\Theta=(\Sigma,\varphi)$ be a complex unit gain graph. Then the following are equivalent.\\
(i) $K^{D}(\Theta)$ is well defined.\\
(ii) $K^{D^{\max}(\Theta)}=K^{D^{\min}(\Theta)}=K^{D^{\max}_<}(\Theta)=K^{D^{\min}_<}(\Theta)=K^{D}(\Theta)$ is balanced.\\
(iii) ${D^{\max}_<}={D^{\min}_<}$, for some vertex ordering $(V(\Sigma),<)$.
\end{theorem}

Now, we characterize the singularity in $\mathbb{T}$-gain distance Laplacian matrices. In \cite{Samanta2022}, Samanta and Kanan proved the following result which characterize balance in $\mathbb{T}$-gain graphs using gain distances.

\begin{theorem}\label{Theroem10}
\cite{Samanta2022} Let $\Theta=(\Sigma,\varphi)$ be a complex unit gain graph with the vertex ordering $(V(\Sigma),<)$. Then the following are equivalent.\\
(i) $\Theta$ is balanced.\\
(ii) $K^{D^{\max}}$ is balanced.\\
(iii) $K^{D^{\min}}$ is balanced.\\
(iv) $D^{\max}(\Theta)=D^{\min}(\Theta)$ and the associated complex unit gain graph $K^{D}(\Theta)$ is balanced.
\end{theorem}

Now we are in the position to characterize singularity.\\

\begin{theorem}\label{Theroem11}
For any connected $\mathbb{T}$-gain $\Theta=(\Sigma,\varphi)$ the following properties are equivalent.\\
(i) The maximum $\mathbb{T}$-gain distance Laplacian determinant $\det(DL^{\max}(\Theta))=0$.\\
(ii) The minimum $\mathbb{T}$-gain distance Laplacian determinant $\det(DL^{\min}(\Theta))=0$.\\
(iii)  $DL^{\max}(\Theta)=DL^{\min}(\Theta)$ and $\det(DL(\Theta))=0$.\\
(iv) $\Theta$ is balanced.\\

Particularly, if $\Theta$ is balanced then the rank of their distance Laplacian matrices is $n-1$, and the rank is $n$, where $\Theta$ is unbalanced.
\end{theorem}

\begin{proof}
Corresponding to the two $\mathbb{T}$-gain complete graphs $K^{D^{\max}}(\Theta)$ and $K^{D^{\min}}(\Theta)$, we define two weighted $\mathbb{T}$-gain complete graphs $(K^{D^{\max}}(\Theta),w)$ and $(K^{D^{\min}}(\Theta),w)$, respectively. Where, we define $w(e_i)=d(v_j,v_k)$ for any edge $e_i=v_jv_k$. Then

$$L(K^{D^{\max}}(\Theta),w)=DL^{\max}(\Theta) ~\mathrm{and}~ L(K^{D^{\min}}(\Theta),w)=DL^{\min}(\Theta).$$

Hence, by Theorem \ref{Theorem6}, $\det(DL^{\max}(\Theta))=0$ if and only if $K^{D^{\max}}(\Theta)$ is balanced, and $\det(DL^{\min}(\Theta))=0$ if and only if $K^{D^{\min}}(\Theta)$ is balanced. Thus, by Theorem \ref{Theroem10} we obtain the required results.\\

The required expression for the rank of those matrices can be followed by Theorem \ref{Theorem4}, and Theorem \ref{Theorem5}.

\end{proof}

Finally, we provide with a spectral characterization of balance, while using the $\mathbb{T}$-gain distance Lapalcian spectrum.\\

A {\it switching function} for a $\mathbb{T}$-gain graph $\Theta=(\Sigma,\varphi)$ is a map $\xi:V(\Sigma)\rightarrow \mathbb{T}$. By {\it switching} a non-empty $\mathbb{T}$-gain graph $\Theta=(\Sigma,\varphi)$ means replacing $\varphi$ by $\varphi^{\xi}$, where $\varphi^{\xi}(e_{jk})=\xi(v_j)^{-1}\varphi(e_{jk})\xi(v_k)$ and $\Theta^{\xi}=(\Sigma,\varphi^{\xi})$ denote a $\mathbb{T}$-gain graph after switching.

\begin{lemma}\label{Lemma5}
Switching a $\mathbb{T}$-gain graph $\Theta=(\Sigma,\varphi)$ does not change the set of compatible pairs of vertices.
\end{lemma}

\begin{proof}
  For any two arbitrary vertices $x$ and $y$, the switching either change or not the gain of $xy$-paths. Thus, if the vertices $x$ and $y$ are compatible in $\Theta(\Sigma,\varphi)$, then also compatible in $\Theta^{\xi}=(\Sigma,\varphi^{\xi})$, and conversely.
\end{proof}

\begin{theorem}\label{Theroem12}
If $\Theta=(\Sigma,\varphi)$ is switched to $\Theta^{\xi}=(\Sigma,\varphi^{\xi})$ and if $D^{\max}(\Theta)=D^{\min}(\Theta)=D(\Theta)$, then $D^{\max}(\Theta^{\xi})=D^{\min}(\Theta^{\xi})=D(\Theta^{\xi})$ and $D(\Theta)$ is similar to $D(\Theta^{\xi})$.
\end{theorem}

\begin{proof}
Let $\Theta=(\Sigma,\varphi)$ switched to $\Theta^{\xi}=(\Sigma,\varphi^{\xi})$ and let $D^{\max}(\Theta)=D^{\min}(\Theta)=D(\Theta)$, then by Lemma \ref{Lemma5} we follow that $D^{\max}(\Theta^{\xi})=D^{\min}(\Theta^{\xi})=D(\Theta^{\xi})$. Next, one can easily calculate the switching matrix $S$ such that $S=\mathrm{diag}(\xi(v_1),\xi(v_2), \cdots , \xi(v_n))=S^{-1}$, which implies $D(\Theta^{\xi})=SD(\Theta)S=SD(\Theta)S^{-1}$.
\end{proof}

By the above theorem, it is clear that the spectra of the distances matrices remain the same after switching and also same in the case of distance Laplacian, which we prove now.

\begin{lemma}\label{Lemma6}
Let $\Theta=(\Sigma,\varphi)$ be a $\mathbb{T}$-gain graph and let $\xi$ be a switching function. Then, the distance Laplacian matrices of $\Theta^{\xi}=(\Sigma,\varphi^{\xi})$ are similar to those of $\Theta=(\Sigma,\varphi)$ and hence having same spectrums.
\end{lemma}

\begin{proof}
Here, we prove for the case $DL^{\max}$, the remaining cases to be left as an exercise. Since $D^{\max}(\Theta^{\xi})=SD^{\max}(\Theta)S^{-1}$, thus we have

\begin{eqnarray*}
 S.DL^{\max}(\Theta)S^{-1} &=& S\bigr(\mathrm{Tr}(\Sigma)-D^{\max}(\Theta)\bigr)S^{-1} \\
   &=& DL^{\max}(\Theta^{\xi}).
\end{eqnarray*}

Thus, $DL^{\max}(\Theta^{\xi})$ is similar to $DL^{\max}(\Theta)$.

\end{proof}

\begin{theorem}\label{Theroem13}
Let $\Theta=(\Sigma,\varphi)$ be a $\mathbb{T}$-gain graph, then $\Theta=(\Sigma,\varphi)$ is balanced if and only if $DL^{\max}(\Theta)=DL^{\min}(\Theta)=DL(\Theta)$ and $DL(\Theta)$ is cospectral to $DL(\Sigma)$, where $DL(\Sigma)$ is the distance Laplacian matrix of the underlying graph $\Sigma$.
\end{theorem}

\begin{proof}
Let $\Theta=(\Sigma,\varphi)$ is balanced, then clearly $\Theta=(\Sigma,\varphi)$ can be switched to a positive $\mathbb{T}$-gain graph $\Theta^{\xi}=(\Sigma,+)$. Since $\Theta=(\Sigma,\varphi)$ is balanced, by Corollary 5.1 \cite{Samanta2022} $D(\Theta)$ exist and thus Lemma \ref{Lemma6} implies $DL(\Theta)$ is similar to $DL(\Sigma,+)=DL(\Sigma)$ and hence $DL(\Theta)$ is cospectral to $DL(\Sigma)$.\\

Conversely, $DL^{\max}(\Theta)=DL^{\min}(\Theta)=DL(\Theta)$ and $DL(\Theta)$ is cospectral to $DL(\Sigma)$. Thus, $\det(DL(\Theta))=\det(DL(\Sigma))=0$, and by Theorem \ref{Theroem11}, $\Theta=(\Sigma,\varphi)$ is balanced.
\end{proof}
\ \
\vspace{0.7cm}

\textbf{Disclosure statement}\\
The author's have no conflict of interest.\\

\ \
\vspace{0.7cm}

\textbf{Acknowledgements}\\
I would like to thank Professor Thomas Zaslavsky (Binghamton University, New York) for reviewing the article before submission to the journal. Because he put some valuable comments and suggestions which improved the quality of the paper. The research is partially supported by the Italian National Group for Algebraic and Geometric Structures and their Applications (GNSAGA - INdAM).\\

\ \

\end{document}